\documentclass[a4paper,12pt]{amsart}
\usepackage{amssymb,amscd,amsmath}
\usepackage{tikz}
\usetikzlibrary{shapes,arrows,calc}

\theoremstyle{plain}
\newtheorem{thm}{Theorem}[section]
\newtheorem{prop}[thm]{Proposition}
\newtheorem{lem}[thm]{Lemma}
\newtheorem{cor}[thm]{Corollary}

\theoremstyle{definition}
\newtheorem{defn}[thm]{Definition}

\theoremstyle{remark}
\newtheorem{rem}{Remark}[thm]
\newtheorem{ex}[rem]{Example}

\newcommand{\Complex}{\mathbb{C}}

\newcommand{\Fp}{{\FF_p}}

\newcommand{\FF}{\mathbb{F}}

\newcommand{\Qed}{\hfill\qed}

\newcommand{\IM}[1]{\mathrm{Im}\,{#1}}

\newcommand{\gn}[1]{{\langle{#1}\rangle}}

\newcommand{\Deq}{\leftarrow}

\begin{document}

\title{Uniqueness of Butson Hadamard matrices of small degrees}

\author{Mitsugu Hirasaka$^1$}
\address{Department of mathematics, Pusan National University, jang-jeon dong, Busan, Republic of Korea}
\email{hirasaka@pusan.ac.kr}

\author{Kyoung-Tark Kim$^2$}
\address{Department of mathematics, Pusan National University, jang-jeon dong, Busan, Republic of Korea}
\email{poisonn00@hanmail.net}

\author{Yoshihiro Mizoguchi}
\address{Institute of Mathematics for Industry, Kyushu University, 744 Motooka, Nishi-ku Fukuoka 819-0395, Japan}
\email{ym@imi.kyushu-u.ac.jp}

\maketitle

\footnotetext[1]{This work was supported by the Financial Supporting Project of Long-term Overseas Dispatch of PNU's Tenure-track Faculty, 2012.}
\footnotetext[2]{This work was supported by Kyushu University Friendship Scolarship.}

\begin{abstract}
  For positive integers $m$ and $n$, we denote by $\mathrm{BH}(m,n)$ the set of all $H\in M_{n\times n}(\Complex)$ such that $HH^\ast=nI_n$ and each entry of $H$ is an $m$-th root of unity where $H^\ast$ is the adjoint matrix of $H$ and $I_n$ is the identity matrix.
  For $H_1,H_2\in \mathrm{BH}(m,n)$ we say that $H_1$ is \textit{equivalent} to $H_2$ if $H_1=PH_2 Q$ for some monomial matrices $P, Q$ whose nonzero entries are $m$-th roots of unity.
  In this paper we classify $\mathrm{BH}(17,17)$ up to equivalence by computer search.
\end{abstract}
\section{Introduction}

Following \cite{butson}, we call an $n\times n$ complex matrix $H$ a \textit{Butson-Hadamard matrix of type $(m,n)$} if each entry of $H$ is an $m$-th root of unity and $HH^\ast = nI_n$ where $H^\ast$ is the conjugate transpose of $H$ and $I_n$ is the $n\times n$ identity matrix.
We denote by $\mathrm{BH}(m,n)$ the set of all Butson-Hadamard matrices of type $(m,n)$.
We give an equivalence relation on $\mathrm{BH}(m,n)$:
$H_1,H_2 \in \mathrm{BH}(m,n)$ are \textit{equivalent} if $H_2$ can be obtained from $H_1$ via a finite sequence of the following operations:
  \begin{enumerate}
  \item [(O1)] a permutation of the rows (columns);
  \item [(O2)] a multiplication of a row (column) by an $m$-th root of unity.
 \end{enumerate}

In this paper we focus on $\mathrm{BH}(p,p)$ where $p$ is a prime.
It is well-known that the \textit{Fourier matrix} $F_p = (\exp\frac{2\pi\sqrt{-1}ij}{p})_{0\leq i,j\leq p-1}$ of degree $p$ is in $\mathrm{BH}(p,p)$ for each prime $p$, but it is still open whether or not every matrix in $\mathrm{BH}(p,p)$ is equivalent to $F_p$.
On the other hand it would be a quite exciting result if we could find a matrix in $\mathrm{BH}(p,p)$ which is not equivalent to $F_p$.
Because, such a matrix gives rise to a non-Desarguesian projective plane of order $p$ (see Proposition \ref{nonfnond}).

One may get a positive answer for the uniqueness of the equivalence classes on $\mathrm{BH}(p,p)$ for $p=2,3,5,7$ without any use of computer, and also for $p=11,13$ with a light support of computer.
(The complexity over 3.0 GHz CPU is about less than $10$ seconds.)
But, for larger prime numbers $p$, one may notice that a heavy amount of complexity is needed in order to classify matrices in $\mathrm{BH}(p,p)$.
In fact it was estimated to take about 5000 hours in order to do it for $\mathrm{BH}(17,17)$ over a single 3.0 GHz CPU.
We introduced a parallel algorithm to solve the following result.
The computation is executed on the high performance multi-node server system Fujitsu Primergy CX400 in Kyushu University.

\begin{thm}\label{mainthm}
For a prime $p \leq 17$, every matrix in $\mathrm{BH}(p,p)$ is equivalent to the Fourier matrix of degree $p$.
\end{thm}

In section 2 we explain our algorithm to find up to equivalence all the matrices in $\mathrm{BH}(p,p)$.
In section 3 we will prove that if there is a matrix in $\mathrm{BH}(p,p)$ which is not equivalent to the Fourier matrix $F_p$ then there exists a non-Desarguesian projective plane of order $p$.

\section{Algorithm to classify $\mathrm{BH}(p,p)$}

Throughout this paper the entries of an $n \times n$ matrix is indexed by integers from $0$ to $n-1$.
For instance, the upper leftmost entry is considered to be in $(0,0)$-position rather than $(1,1)$-position, and lower rightmost entry is in $(n-1,n-1)$-position than $(n,n)$-position.

In the sequel we assume that $p$ is prime and
\[\xi_p=\cos(2\pi/p)+\sqrt{-1}\sin(2\pi/p).\]
We denote by $\Fp = \{0, 1, \ldots, p-1 \}$ a finite field with $p$ elements, and adopt the natural ordering of $\Fp$, i.e., $0 < 1 < \cdots < p-1$.

\begin{defn}
  We say that $D = (D_{i,j})\in M_{p\times p}(\Fp)$ is a \textit{difference matrix} if $\Fp = \{ D_{i,k} - D_{j,k} \mid k = 0,1, \ldots, p-1 \}$ for any $i$ and $j$ with $i \neq j$.
  The set of all difference matrices of degree $p$ is denoted by $\mathcal{D}(p)$.
\end{defn}

We define a map $\lambda : \mathrm{BH}(p,p) \rightarrow M_{p\times p}(\Fp)$ by $\lambda(H) = (E_{i,j})$ for $H = (\xi_p^{E_{i,j}}) \in \mathrm{BH}(p,p)$.
(Since $(\xi_p)^p = 1$ we can regard an exponent $E_{i,j}$ as an element of $\Fp$.)

\begin{lem}\label{lem:bi}
  The map $\lambda$ is one to one and $\IM{\lambda} = \mathcal{D}(p)$.
  So there is a one to one correspondence between $\mathrm{BH}(p,p)$ and $\mathcal{D}(p)$.
\end{lem}

\begin{proof}
The injectivity follows from the definition of $\lambda$.
Let $H = (\xi_p^{E_{i,j}}) \in \mathrm{BH}(p,p)$.
Then, for all distinct $i,j$ with $0 \leq i,j\leq p-1$,
\[ (HH^\ast)_{i,j} = \sum_{k=0}^{p-1} H_{i,k} \bar{H}_{j,k} = \sum_{k=0}^{p-1} \xi_p^{E_{i,k}-E_{j,k}}.\]
Since $x^{p-1} + \cdots + x + 1$ is the minimal polynomial of $\xi_p$, $(HH^\ast)_{i,j} = 0$ if and only if $\{E_{i,k}-E_{j,k}\mid k=0,1,\ldots, p-1\}=\Fp$.
Hence $\lambda(H) \in \mathcal{D}(p)$ and $\lambda$ is onto $\mathcal{D}(p)$.
\end{proof}

For $D = (D_{i,j}) \in \mathcal{D}(p)$ we say that $D$ is \textit{fully normalized} if $D_{0,i} = D_{i,0} = 0$ and $D_{1,i} = D_{i,1} = i$ for all $i=0,1,\ldots, p-1$.
For $H \in \mathrm{BH}(p,p)$, $H$ is called \textit{fully normalized} if so is $\lambda(H)$.
If $N = (N_{i,j})$ in $\mathcal{D}(p)$ (in $\mathrm{BH}(p,p)$, respectively) is fully normalized then the $(p-2) \times (p-2)$ submatrix $(N_{i,j})_{2 \leq i,j \leq p-1}$ is called the \textit{core} of $N$.

Classifying $\mathrm{BH}(p,p)$ is equivalent to finding all possible cores of fully normalized matrices in $\mathrm{BH}(p,p)$.
For convenience we can move our workspace to $\mathcal{D}(p)$ due to Lemma \ref{lem:bi}.
The next proposition shows that there is a systematic way to find a difference matrix:

\begin{prop}\label{sndiff}
  Let $L = (L_{i,j}) \in M_{p\times p}(\Fp)$.
  Then, $L \in \mathcal{D}(p)$ if and only if $L_{i,j}\neq L_{i,b} + L_{a,j} - L_{a,b}$ for all $0\leq a<i \leq p-1$ and $0\leq b<j \leq p-1$.
\end{prop}

\begin{proof}
  ($\Rightarrow$)
  By the definition of a difference matrix we have $L_{i,j} - L_{a,j}\neq L_{i,b}  - L_{a,b}$.
  ($\Leftarrow$)
  Fix $i$ and $a$.
  Then $\{ L_{i,k} - L_{a,k} \mid k = 0,\ldots, p-1 \} = \Fp$ by the condition.
\end{proof}

Fix $i$ and $j$ with $0< i,j \leq p-1$.
Then Proposition \ref{sndiff} tells us that if we hope to determine the $(i,j)$-entry of a difference matrix then we have to check the condition $L_{i,j}\neq L_{i,b} + L_{a,j} - L_{a,b}$ for all $a$ and $b$ with $0 \leq a < i$ and $0 \leq b < j$.
This leads the following algorithm:

\vskip5mm

Algorithm, $C(i,j)$:

\vskip3mm

Input: $i,j \in \{1, \ldots, p-1\}$ and a $p\times p$ matrix $L = (L_{i,j})$

Output: $r(i,j)$ (a subset of $\Fp$)

\vskip3mm

$r(i,j)\Deq\Fp$; $a\Deq 0$; $b\Deq 0$

WHILE $0\leq a<i$ DO

\qquad WHILE $0\leq b<j$ DO

\qquad\qquad $r(i,j)\Deq r(i,j)\setminus\{L_{i,b} + L_{a,j} - L_{a,b}\}$

\qquad\qquad $b \Deq b+1$

\qquad $a \Deq a+1$

RETURN $r(i,j)$

\vskip5mm

The algorithm $C(i,j)$ returns a set $r(i,j)$ of candidates for the entry $L_{i,j}$ if the upper left entries $L_{a,b}$ ($0\leq a<i$ and $0\leq b<j$) are already determined.

Now suppose that we hope to construct a fully normalized matrix in $\mathcal{D}(p)$.
Let $L$ be a matrix in $M_{p\times p}(\Fp \cup \{\bot\})$ such that
\begin{equation}
  \label{eq:1}
  L_{0,i} = L_{i,0} = 0,\: L_{1,i} = L_{i,1} = i \mbox{ and } L_{j,k} = \bot
\end{equation}
for all $i \in \{0, \ldots, p-1\}$ and $2 \leq j,k \leq p-1$ where $\Fp \cap \{\bot\} = \emptyset$.
(The letter `$\bot$' stands for the `empty' entry.)
In the sequel we should fill the core of $L$ by using the algorithm $C(i,j)$ so that $L \in \mathcal{D}(p)$.
First of all we need an appropriate order of computation which is compatible to the algorithm $C(i,j)$:

\begin{defn}\label{properorder}
  Let $\mathcal{I} = \{ (i,j) \mid 2 \leq i,j \leq p-1 \}$ be the set of indices of the core of $L$.
  A total order $\preceq$ on $\mathcal{I}$ is called \textit{admissible} if the following conditions hold.
  \begin{enumerate}
  \item For all $(i,j) \in \mathcal{I}$ we have $(2,2) \preceq (i,j)$;
  \item For any $(i,j) \in \mathcal{I}$, if $2 \leq k \leq i$, $2\leq l \leq j$ then $(k,l) \preceq (i,j)$.
  \end{enumerate}
\end{defn}

\begin{ex}\label{orderex}
  The following are admissible total orders on $\mathcal{I}$.
  \begin{enumerate}
  \item Diagonal order 1, $\preceq_{D}$: $(2,2) \prec (2,3) \prec (3,2) \prec (3,3) \prec (2,4) \prec (3,4) \prec (4,2) \prec (4,3) \prec (4,4) \prec \cdots$.
  \item Diagonal order 2, $\preceq_{D'}$: $(2,2) \prec (2,3) \prec (3,2) \prec (2,4) \prec (3,3) \prec (4,2) \prec (2,5) \prec (3,4) \prec (4,3) \prec (5,2) \prec \cdots$.
  \item Horizontal order, $\preceq_H$: $(2,2) \prec (2,3) \prec \cdots \prec (2,p-1) \prec (3,2) \prec \cdots \prec (3,p-1) \prec (4,2) \prec \cdots$.
  \end{enumerate}
\end{ex}

With an admissible total order $\preceq$ on $\mathcal{I}$ we now introduce the main algorithm $M(a,b,c,d)$.
See Figure \ref{fig:1}.
Notice that the parameter $(a,b)$ (respectively, $(c,d)$) indicates the starting (resp. finishing) index of the algorithm.
For example, by calling $M(2,2,p-1,p-1)$, we can obtain all possible cores of fully normalized matrices in $\mathcal{D}(p)$.

\begin{figure}
  \centering
  \tikzstyle{dec} = [diamond, draw, aspect = 2, text width=18mm, text badly centered, inner sep=0pt, font=\tiny]
  \tikzstyle{rblock} = [rectangle, draw, text centered, rounded corners=3mm, minimum height=6mm, font=\tiny]
  \tikzstyle{block} = [rectangle, draw, text centered, minimum height=6mm, font=\tiny]
  \tikzstyle{line} = [draw, -latex']
  \tikzstyle{Edge} = [black]
  \tikzstyle{Arr} = [black,-latex']
  \begin{tikzpicture}[auto]
    \node [rblock] (init) {Start};
    \node [block, right of=init, node distance = 50mm] (initij) {$(i,j) \Deq (a,b)$};
    \node [block, below of=initij, node distance = 10mm] (p1) {Call $C(i,j)$, $A(i,j) \Deq r(i,j)$};
    \node [dec, below of=p1, node distance = 16mm] (d1) {$A(i,j)=\emptyset?$};
    \node [block, below of=d1, node distance = 13mm] (p2) {$L_{i,j} \Deq$smallest element in $A(i,j)$};
    \node [dec, below of=p2,node distance = 12mm] (d2) {$(i,j) = (c,d)?$};
    \node [block, left of=d2, node distance = 40mm] (p3) {$(i,j)\Deq$successor of $(i,j)$};
    \node [block, below of=d2, node distance = 13mm] (p4) {Save entries $L_{k,l}$ with $(a,b) \preceq (k,l) \preceq (c,d)$};
    \node [block, below of=p4, node distance = 10mm] (p5) {$(i,j)\Deq$predecessor of $(i,j)$};
    \node [block, left of=p5, node distance = 42mm] (p6) {$A(i,j)\Deq A(i,j)\setminus\{L_{i,j}\}$};
    \node [dec, right of=d1, node distance = 33mm] (d3) {$(i,j) = (a,b)?$};
    \node [rblock, above of=d3, node distance = 25mm] (end) {End};

    \path [line] (init) -- (initij);
    \path [line] (initij) -- (p1);
    \path [line] (p1) -- (d1);
    \path [line] (d1) -- node [near start] {\footnotesize No} (p2);
    \path [line] (p2) -- (d2);
    \path [line] (d2) -- node [near start] {\footnotesize No} (p3);
    \path [line] (p3) |- (p1);
    \path [line] (d2) -- node [near start] {\footnotesize Yes} (p4);
    \path [line] (p4) -- (p5);
    \path [line] (p5) -- (p6);
    \path [line] (d1) -- node [near start] {\footnotesize Yes} (d3);
    \path [line] (d3) -- node [near start] {\footnotesize Yes} (end);
    \path [line] (d3) |- node [near start] {\footnotesize No} (p5);

    \draw ($(p6.west)$) edge[Edge] ($(p6.west) + (-1mm,0)$);
    \draw ($(p6.west) + (-1mm,0)$) edge[Edge] ($(p6.west) + (-1mm,0) + (0,48mm)$);
    \draw ($(p6.west) + (-1mm,0) + (0,48mm)$) edge[Arr] ($(d1.west)$);
  \end{tikzpicture}
  \caption{The main algorithm, $M(a,b,c,d)$.}
  \label{fig:1}
\end{figure}

There is a redundancy in our algorithm.
Notice that if there exists a matrix $A$ in $\mathcal{D}(p)$ then the transpose $A^{\mathrm{T}}$ is also in $\mathcal{D}(p)$, because the initial part (cf. the equation \eqref{eq:1}) of the construction for $L$ is symmetric.
Although $A$ and $A^{\mathrm{T}}$ may not be equivalent it is sufficient to find only one of $A$ and $A^{\mathrm{T}}$ in the searching algorithm, and we just add each transpose to the result in the final step.
Therefore we may assume
\begin{equation}
  \label{eq:2}
  L_{2,3}\leq L_{3,2}.
\end{equation}

For primes $p \leq 13$ the main algorithm $M(2,2,p-1,p-1)$ works well.
Over 3.0 GHz CPU within less than 10 seconds, we obtain the following result:
For a prime $p \leq 13$, there is a unique fully normalized matrix in $\mathrm{BH}(p,p)$, namely, the Fourier matrix of degree $p$.

The next case $p=17$ needs a heavy computer calculation.
So we use a parallel algorithm to use a supercomputer.
Our strategy is given as follows:
Let $(r,s)$ be a fixed index among a total order $\preceq$.
The \textit{master thread} carries out $M(2,2,r,s)$.
If there is a \emph{partial solution} from $(2,2)$ to $(r,s)$ then the master process passes this partial information of the matrix $L$ to one of many \textit{slave threads}.
For given data from the master thread, a slave thread decides whether or not there are fully normalized matrices in $\mathcal{D}(p)$ by calling $M(m,n,p-1,p-1)$ where $(m,n)$ is the successor of $(r,s)$.
Of course, in our parallel program, the master thread also has the role of \textit{jobs scheduler}, i.e., the management of slave threads.

\begin{figure}
  \centering
  \includegraphics[scale=0.4]{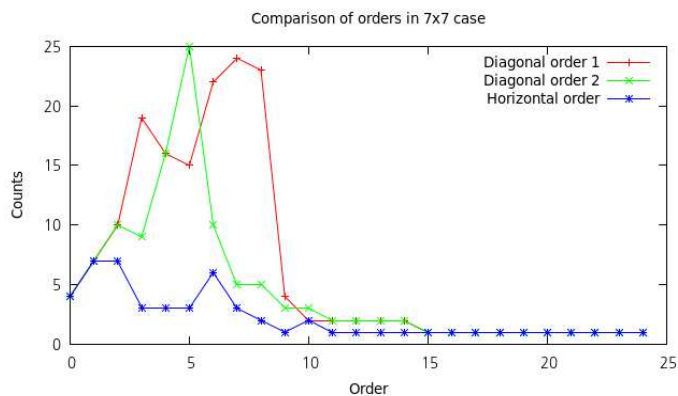}
  \caption{The case $p = 7$.}
  \label{fig:2}
\end{figure}

\begin{figure}
  \centering
  \includegraphics[scale=0.4]{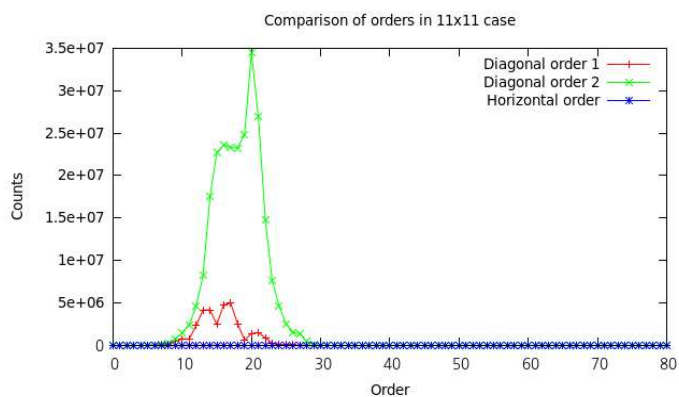}
  \caption{The case $p = 11$.}
  \label{fig:3}
\end{figure}

A choice of the \textit{dividing index} $(r,s)$ (i.e., the finishing index of the master thread) depends on the specific total order $\preceq$.
We checked the three types of total orders, that is, $\preceq_D, \preceq_{D'}$ and $\preceq_H$.
(See Example \ref{orderex}.)
The figure \ref{fig:2} and \ref{fig:3} show respectively the cases of $p = 7$ and $p = 11$.
The X-axis of the figures stands for choices of the dividing indices $(r,s)$ among total orders, and the Y-axis means the corresponding counts of possibility for the partial results which is carried out by $M(2,2,r,s)$.
We see that the horizontal order is most efficient in the three types.
Therefore we adopt the horizontal order in the case of $p = 17$ too, and in this case we choose the dividing index as $(2,16)$ as Figure \ref{fig:4} suggested.

\begin{figure}
  \centering
  {\tiny
    \begin{tabular}{|c|r|r|}
      \hline
      $(r,s)$ & $M(2,2,r,s)$ & $\#$Partial results \\ \hline \hline
      $(2,2)$ & {$\varepsilon$} seconds & 14 \\ \hline
      $(2,3)$ & {$\varepsilon$} seconds & 157 \\ \hline
      $(2,4)$ & {$\varepsilon$} seconds & 1507 \\ \hline
      $(2,5)$ & {$\varepsilon$} seconds & 12327 \\ \hline
      $(2,6)$ & {$\varepsilon$} seconds & 84573 \\ \hline
      $(2,7)$ & {$\varepsilon$} seconds & 478501 \\ \hline
      $(2,8)$ & 1 seconds & 2186161 \\ \hline
      $(2,9)$ & 1 seconds & 7865605 \\ \hline
      $(2,10)$ & 5 seconds & 21644469 \\ \hline
      $(2,11)$ & 12 seconds & 43828409 \\ \hline
      $(2,12)$ & 29 seconds & 61675825 \\ \hline
    \end{tabular}\hskip5mm
    \begin{tabular}{|c|r|r|}
      \hline
      $(r,s)$ & $M(2,2,r,s)$ & $\#$Partial results \\ \hline \hline
      $(2,13)$ & 50 seconds & 55494757 \\ \hline
      $(2,14)$ & 69 seconds & 28008069 \\ \hline
      $(2,15)$ & 81 seconds & 6275119 \\ \hline
      $(2,16)$ & 81 seconds & 6275119 \\ \hline
      $(3,2)$ & 85 seconds & 37464544 \\ \hline
      $(3,3)$ & 112 seconds & 376242051 \\ \hline
      $(3,4)$ & 335 seconds & 2737088388 \\ \hline
      $(3,5)$ & 1852 seconds & 15753030361 \\ \hline
      $(3,6)$ & 9878 seconds & 71394611311 \\ \hline
      $\vdots$ & $\vdots$ & $\vdots$ \\
    \end{tabular}
  }
  
  \caption{The computation data in the case $p = 17$.}
  \label{fig:4}
\end{figure}

The specification of parallel computation for $p = 17$ is the following:

\vskip3mm
Fujitsu PRIMERGY CX400 \cite{TATARA};

CPU: Intel Xeon E5-2680 (2.7GHz, 8core) $\times$ 2 / node;

Memory: 128GB / node

Interconnection network: InfiniBand FDR×1 6.78GB/sec

Server system total peak performance: 811.86TFLOPS (1476 nodes)

OS: Red Hat Enterprise Linux;

Programming language: C with MPI (message passing interface);

Total number of processes: 1 (master) + 63 (slaves) = 64;

Total required time: 246093 seconds ($\doteqdot$ 68 hours);
\vskip3mm

As mentioned in introduction, we obtain Theorem \ref{mainthm} as a result.

\section{Desarguesian projective plane yields the Fourier matrix.}

Let $\mathcal{A}$ be a nonempty finite set and $\mathcal{B}$ a family of subsets of $\mathcal{A}$.
We say that $\rho \in \mathrm{Sym}(\mathcal{A}\cup\mathcal{B})$ is an \textit{automorphism} of $(\mathcal{A},\mathcal{B})$ if, for all $(a,B)\in \mathcal{A}\times \mathcal{B}$, $a\in B$ if and only if $\rho(a)\in \rho(B)$.
We denote by $\mathrm{Aut}(\mathcal{A},\mathcal{B})$ the group of automorphisms of $(\mathcal{A},\mathcal{B})$.

For a positive integer $k \geq 2$ a pair $\mathcal{D} = (\mathcal{P},\mathcal{L})$ is called a \textit{projective plane} of \textit{order} $k$ if $|\mathcal{P}| = |\mathcal{L}| = k^2+k+1$, $|\{x\in \mathcal{P}\mid x\in L\}| = k+1$ for each $L\in \mathcal{L}$ and $|\{x\in \mathcal{P}\mid x\in L\cap L'\}| = 1$ for all distinct $L,L'\in \mathcal{L}$.
A pair $(x,L) \in \mathcal{P}\times\mathcal{L}$ is called a \textit{flag} of $\mathcal{D}$ if $x\in L$.
For a flag $(x,L)$ of $\mathcal{D}$ we say that $\sigma\in \mathrm{Aut}(\mathcal{P},\mathcal{L})$ is an \textit{elation} with respect to $(x,L)$ if $\sigma$ fixes each point in $L$ and each line through $x$.

Let $\mathcal{D}=(\mathcal{P},\mathcal{L})$ be a projective plane of order $p$ containing an elation $\sigma$ of order $p$ with respect to a flag $(x,L)$.
Let $y,z\in \mathcal{P}\setminus L$ be such that $x$, $y$ and $z$ are not on a common line.
For $i \in \{0,1,\ldots, p-1\}$ we define $N_i\in \mathcal{L}$ to be the line through $y$ and $\sigma^i(z)$, and $y_i\in \mathcal{P}$ to be the point incident to $N_0$ and $\sigma^{-i}(N_1)$.

\begin{lem}\label{lem:10}
  For all $i,j \in \{0,\ldots,p-1\}$ there is a unique $E_{i,j}\in \Fp$ such that $\sigma^{E_{i,j}}(y_i)\in N_j$.
  Moreover $(E_{i,j})$ is fully normalized in $\mathcal{D}(p)$.
\end{lem}

\begin{proof}
  Since $y_i\in \mathcal{P}\setminus L$ and $x\notin N_j$, the line $M$ through $x$ and $y_i$ intersects $N_j$ at exactly one point.
  Since $\sigma$ acts regularly on $M\setminus\{x\}$, the first assertion follows.
  Since $y_i\in N_0$ and $y=y_0\in N_i$ we have $E_{i,0} = E_{0,i} = 0$ for each $i \in \{0,1,\ldots, p-1\}$.
  Since $y_i \in \sigma^{-i}(N_1)$ and $\sigma^i(y_1)=\sigma^i(z)\in N_i$ we have $\sigma^i(y_i)\in N_1$ and $\sigma^i(y_1)\in N_i$ whence $E_{i,1} = E_{1,i} = i$ for each $i \in \{0,1,\ldots, p-1\}$.
  Suppose that $E_{i,k} - E_{j,k}$ ($k=0,1,\ldots,p-1$) are not distinct for some $i\ne j$, i.e., $E_{i,k} - E_{j,k} = E_{i,l} - E_{j,l}$ for some $k\ne l$.
  Since $\sigma^{E_{i,k}}(y_i), \sigma^{E_{j,k}}(y_j)\in N_k$ and $\sigma^{E_{i,l}}(y_i), \sigma^{E_{j,l}}(y_j)\in N_l$ it follows that
  \[\sigma^{E_{i,k} - E_{j,k}}(y_i), y_j\in \sigma^{-E_{j,k}}(N_k) \cap \sigma^{-E_{j,l}}(N_l).\]
  Since $y_j\notin \gn{\sigma}y_i$ and $N_k\notin \gn{\sigma}N_l$ we have a contradiction.
  This completes the proof of the second assertion.
\end{proof}

We fix the point set $\mathcal{P}$ of size $p^2+p+1$.
Let $\Delta$ be the set of all quadruples $(\mathcal{D},\sigma,y,z)$ satisfying the following conditions:
\begin{enumerate}
\item $\mathcal{D}=(\mathcal{P},\mathcal{L})$ is a projective plane of order $p$;
\item $\sigma$ is an elation of $\mathcal{D}$ with respect to a flag $(x,L)$;
\item $y,z\in \mathcal{P}\setminus L$ such that $x,y,z$ are not on a common line.
\end{enumerate}
By Lemma~\ref{lem:10} we define a function $\Psi$ from $\Delta$ to the set of all fully normalized Butson-Hadmard matrices of type $(p,p)$ by $\Psi(\mathcal{D},\sigma,x,y) = (\xi_p^{E_{i,j}})$ where $\sigma^{E_{i,j}}(y_i)\in N_j$.

\begin{lem}\label{lem:20}
  The function $\Psi$ is surjective.
\end{lem}

\begin{proof}
  Let $H\in \mathrm{BH}(p,p)$ be fully normalized.
  Then $H=(\xi_p^{E_{i,j}})$ where $(E_{i,j}) \in \mathcal{D}(p)$ is also fully normalized.
  Let $C$ denote the $p\times p$ permutation matrix corresponding to the map from $\Fp$ to itself defined by $\alpha\mapsto \alpha + 1$.
  We denote the $p^2\times p^2$ matrix $(C^{E_{i,j}})$ by $P(H)$.
  We denote the $m\times n$ all one and zero matrix by $J_{m,n}$ and $O_{m,n}$ respectively, and we define $Q(H)$ to be a $(p^2 + p + 1)\times (p^2 + p + 1)$ matrix such that
  \[Q(H)=\left(
    \begin{array}{c|c|c}
      1 & J_{1,p} & O_{1,p^2}\\
      \hline
      J_{p,1} & O_{p,p} & D\\
      \hline
      O_{p^2,1} & D^T & P(H)
    \end{array}\right)\]
  where $D$ is a $p \times p^2$ matrix and
  \[D=\left(
    \begin{array}{c|c|c|c}
      J_{1,p} & O_{1,p} & \cdots & O_{1,p}\\
      \hline
      O_{1,p} & J_{1,p} & \ddots & \vdots \\
      \hline
      \vdots & \ddots & \ddots & O_{1,p}\\
      \hline
      O_{1,p} & \cdots & O_{1,p} & J_{1,p}\\
    \end{array}\right). \]
  Note that $Q(H)$ forms an incidence matrix of a projective plane of order $p$ and
  \[RQ(H)R^t=Q(H)\quad\mbox{where}\quad R=\left(
    \begin{array}{c|c}
      I_{p+1} &  O_{p+1,p^2}\\
      \hline
      O_{p^2,p+1} & I_p\otimes C
    \end{array}\right).\]
  This implies that the projective plane $\mathcal{D}$ having its incidence matrix $Q(H)$ has an elation $\sigma$ with respect to the flag corresponding to the $(0,0)$-entry of $Q(H)$.
  Let $y,z$ be the points corresponding to the $(p+1)$-th row and $(2p+1)$-th row of $Q(H)$, respectively.
  Then the quadruple $(\mathcal{D},\sigma,y,z)$ is mapped to $H$ by $\Psi$.
  Therefore $\Psi$ is surjective.
\end{proof}

\begin{lem}\label{prop:30}
  If $\mathcal{D}$ is a Desarguesian projective plane of order $p$ then $\Psi(\mathcal{D},\sigma,y,z)$ is $(\xi_p^{ij})$, namely, the Fourier matrix of degree $p$.
\end{lem}

\begin{proof}
  Suppose $\mathcal{D}=(\mathcal{P},\mathcal{L})$ is Desarguesian.
  Then the automorphism group of $\mathcal{D}$ is isomorphic to $\mathrm{PGL}(3,p)$.
  Let $\sigma$ be an elation of order $p$ with respect to a flag $(x,L)$ and let $y,z\in \mathcal{L}$ be such that $x,y,z$ are not in a common line.
  We denote by $G$ the normalizer of $\gn{\sigma}$ in $\mathrm{Aut}(\mathcal{P},\mathcal{L})$.
  It is known that $G$ acts doubly transitively on $\mathcal{P}\setminus L$ and $G\simeq \mathrm{AGL}(2,p)$, and hence $G_{y,z}\simeq \mathrm{AGL}(1,p)$ where we denote by $G_{y,z}$ the stabilizer subgroup fixing $y$ and $z$.
  Note that $G_{y,z}$ contains $\tau$ which acts regularly on $\{y_i\mid i=1,2,\ldots,p-1\}$ and regularly on $\{N_i\mid i=1,2,\ldots,p-1\}$.

  Suppose $\tau(N_1)=N_j$ for some $j$.
  Since $\sigma(y_1)\in N_1$ by the assumption and $y_1=z$,
  \[(\tau \sigma \tau^{-1})(y_1)\in \tau(N_1)=N_j.\]
  Since $\sigma^j(y_1)\in N_j$ by the assumption and $\tau\sigma\tau^{-1}(y_1)\in N_j$ it follows that $\tau \sigma \tau^{-1}=\sigma^j$.
  Since $\tau(y_i)=y_i$ and $\sigma^i(y_i)\in N_1$ we have
  \[\tau\sigma^i\tau^{-1}(y_i)\in \tau(N_1)=N_j.\]
  On the other hand, since $\tau\sigma\tau^{-1}=\sigma^j$ we have $\tau\sigma^i\tau^{-1}=\sigma^{ij}$.
  Thus we have
  \[\sigma^{ij}(y_i)=\tau\sigma^i \tau^{-1}(y_i)\in N_j.\]
  This implies that we have $E_{i,j} = ij$ for all $i$ and $j$.
  This completes the proof.
\end{proof}

\begin{prop}\label{nonfnond}
  If there is a fully normalized matrix in $\mathrm{BH}(p,p)$ which is not the Fourier matrix then it induces a non-Desarguesian projective plane of order $p$.
\end{prop}

\begin{proof}
  This is due to the contrapositive of Lemma \ref{prop:30}.
\end{proof}

From Theorem \ref{mainthm} we have the following result:

\begin{cor}
  For a prime $p\leq 17$, there is no non-Desarguesian projective plane of order $p$.\Qed
\end{cor}

\section*{Acknowledgements}

\noindent 1. The computation was mainly carried out using the computer facilities at Research Institute for Information Technology, Kyushu University \cite{TATARA}.

\noindent 2. We would like to thank Dong-been Kim, a student of the R\&E (research and education) lecture given by the first author for his basic algorithm to search Butson Hadamard matrices.

\end{document}